\documentclass[a4paper,11pt,reqno]{amsart}

\usepackage{amsmath}
\usepackage{amsthm}
\usepackage{amssymb}
\usepackage{graphicx}
\usepackage{epsfig}
\usepackage[pdftex]{color}
 \definecolor{MyRed}{rgb}{0.9,0,0}
 \definecolor{MyGreen}{rgb}{0,0.9,0}
 \definecolor{MyBlue}{rgb}{0,0,0.9}
\usepackage[pdftex]{hyperref}
\hypersetup{breaklinks=true,colorlinks=true,citecolor=MyGreen,urlcolor=MyRed,linkcolor=MyBlue,pdfpagemode=UseNone}
\theoremstyle{plain}

\newtheorem{theorem}{Theorem}

\newtheorem*{thm*}{Theorem}
\newtheorem*{theorem*}{Theorem}

\newtheorem*{claim*}{Claim}
\newtheorem{lemma}[theorem]{Lemma}

\newtheorem{conjecture}[theorem]{Conjecture}

\theoremstyle{definition}

\renewcommand{\leq}{\leqslant}
\renewcommand{\geq}{\geqslant}

\def\COMMENT#1{}
\let\COMMENT=\footnote
\def\GL{\emph{GL}}

\begin{document}

\title{Influences of monotone Boolean functions}
\author{Demetres Christofides}
\email{christod@maths.bham.ac.uk}
\thanks{Supported by the EPSRC, grant no. EP/E02162X/1.}
\address{School of Mathematics, University of
Birmingham, Edgbaston, Birmingham, B15 2TT, UK}
\date{\today}
\subjclass[2000]{05D05; 06E30}
\keywords{Discrete cube; Boolean
functions; Influence}
\begin{abstract}
Recently, Keller and Pilpel conjectured that the influence of a
monotone Boolean function does not decrease if we apply to it an
invertible linear transformation. Our aim in this short note is to
prove this conjecture.
\end{abstract}
\maketitle

\section{Introduction}

Given a positive integer $n$, a \emph{Boolean function} on $n$
variables is a function $f: \{0,1\}^n \to \{0,1\}$. The function is
called \emph{monotone} if for all $x = (x_1,\ldots,x_n), y =
(y_1,\ldots,y_n) \in \{0,1\}^n$ satisfying $x_i \leq y_i$ for each
$1 \leq i \leq n$, we have $f(x) \leq f(y)$.

For an $n$-variable Boolean function $f$, the \emph{influence of the
$i$-th variable on $f$} is defined to be
\[
I_i(f) = \frac{1}{2^n} \sum_{x \in \{0,1\}^n}\left|
f(x+e_i) - f(x) \right|,
\]
where $e_i$ denotes the element of $\{0,1\}^n$ whose only non-zero
coordinate is in the $i$-th position, and addition is done
coordinate-wise modulo two. The \emph{total influence of $f$} is
defined to be
\[
I(f) = \sum_{i=1}^n I_i(f).
\]
For the proof of our result it will be convenient to introduce the
following definition: Given $y \in \{0,1\}^n$ we define the
\emph{influence of $y$ on $f$} to be
\[
I_y(f) = \frac{1}{2^n} \sum_{x \in \{0,1\}^n}\left| f(x+y) - f(x)
\right|.
\]
[We remark that if we consider the correspondence between the
elements of $\{0,1\}^n$ and the subsets of $\{1,\ldots,n\}$ then the
influence of $y$ on $f$ is not the same as the usual definition of
the influence of the set $Y$ (corresponding to $y$) over $f$. Since
we will not be using the latter definition, we hope that no
confusion arises.]

The notion of influence of a variable on a Boolean function was
introduced by Ben-Or and Linial~\cite{BenOr&Linial90}. It has since
found many application in discrete mathematics, theoretical computer
science and social choice theory. We refer the reader
to~\cite{Kalai&Safra06} for a survey of some of these applications.
In this note we study the effect on the influence after applying an
invertible linear transformation on a monotone Boolean function.

Given an $n$-variable Boolean function $f$ and an invertible linear
transformation $L \in \GL_n(\mathbb{F}_2)$, the function $Lf$ is
defined by $Lf(x) = f(Lx)$. In~\cite{Keller&Pilpel09} Keller and
Pilpel raised the following conjecture.

\begin{conjecture}[Keller and Pilpel~\cite{Keller&Pilpel09}]
If $f$ is an $n$-variable monotone Boolean function and $L \in
GL_n(\mathbb{F}_2)$ then $I(f) \leq I(Lf)$.
\end{conjecture}

We prove this conjecture in the next section.

\section{Proof of the conjecture}

To prove the conjecture we will use the following simple
combinatorial lemma. We will prove the lemma using the well-known
Hall's marriage theorem. One can obtain short proofs of the lemma
using other equivalent statements.

\begin{lemma}
Let $L \in GL_n(\mathbb{F}_2)$. Then we can permute the columns of
$L$ to obtain a new matrix $L'$ whose diagonal entries are non-zero.
\end{lemma}

\begin{proof}
Let us define a bipartite graph $B$ on $\{r_1,\ldots,r_n\} \times
\{c_1,\ldots c_n\}$ by joining $r_i$ to $c_j$ if and only if
$L_{ij}=1$. It is enough to prove that $B$ contains a perfect
matching. Indeed, if $r_i$ is matched to $c_{\pi(i)}$, then
$\pi^{-1}$ provides the required permutation of the columns. The
existence of this perfect matching is an immediate consequence of
Hall's marriage theorem. Indeed, if this is not the case, then there
is a set $R$ of $k$ rows and a set $C$ of $\ell < k$ columns such
that every row of $R$ has a non-zero entry only in a column of $C$.
But since $\ell < k$, the rows in $R$ are linearly dependent
contradicting the fact that $L$ is invertible.
\end{proof}

It is immediate that if $L'$ is obtained from $L$ be permuting its
columns then $I(Lf) = I(L'f)$. Indeed, if $L' = LP$ where $P$ is the
permutation matrix which maps $e_i$ to $e_{\pi(i)}$, then $I_i(L'f)
= I_{\pi(i)}(Lf)$ and so the total influences are equal. Thus to
prove the conjecture we may assume by the previous lemma that each
diagonal entry of $L$ is non-zero. In this case, we will prove the
stronger assertion that $I_i(f) \leq I_i(Lf)$ for each $1 \leq i
\leq n$. We claim that $I_i(Lf) = I_{Le_i}(f)$. Indeed,
\begin{align*}
I_i(Lf) &= \frac{1}{2^n} \sum_{x}\left| Lf(x+e_i) -
Lf(x) \right| \\
&= \frac{1}{2^n} \sum_{x}\left| f(Lx+Le_i) - f(Lx)
\right| \\
&= \frac{1}{2^n} \sum_{y}\left| f(y+Le_i) - f(y)
\right| \\
&= I_{Le_i}(f).
\end{align*}
Splitting the sum in the definition of $I_{Le_i}(f)$ into two parts
depending on whether the $i$-th coordinate is equal to zero or not
we obtain that
\begin{align*} I_{Le_i}(f) &= \frac{1}{2^n} \sum_{y}\left|
f(y+Le_i)
- f(y) \right| \\
&= \frac{1}{2^n} \left( \sum_{\{y : y_i=0\}} \left|f(y+Le_i) - f(y)
\right| +  \sum_{\{y : y_i=1\}} \left| f(y) - f(y + Le_i) \right|
\right) \\
&= \frac{1}{2^n} \sum_{\{z:z_i=0\}} \left( \left|f(z+Le_i) - f(z)
\right| + \left|f(z+e_i) - f(z+e_i + Le_i) \right| \right)  \\
&\geq \frac{1}{2^n} \sum_{\{z:z_i=0\}} \left| f(z+e_i) + f(z+Le_i) -
f(z) - f(z+e_i + Le_i)\right|.
\end{align*}
Observe that since each diagonal entry of $L$ is non-zero, the
$i$-th coordinate of $Le_i$ is equal to one and so if the $i$-th
coordinate of $z$ is zero, then the $i$-th coordinate of
$z+e_i+Le_i$ is also zero and so by the monotonicity of $f$ we have
$f(z) \leq f(z+e_i)$ and $f(z+e_i+ Le_i) \leq f(z+Le_i)$. It follows
that
\begin{align*}
I_{Le_i}(f) &\geq \frac{1}{2^n} \sum_{\{z:z_i=0\}} \left| f(z+e_i) +
f(z+Le_i) - f(z)
- f(z+e_i + Le_i)\right| \\
&= \frac{1}{2^n} \sum_{\{z:z_i=0\}} \left| f(z+e_i) - f(z) \right| +
\frac{1}{2^n} \sum_{\{z:z_i=0\}} \left|
f(z+Le_i) - f(z+e_i + Le_i)\right|\\
&= \frac{1}{2^n} \sum_{\{z:z_i=0\}} \left| f(z+e_i) - f(z) \right| +
\frac{1}{2^n} \sum_{\{w:w_i=1\}} \left|f(w) - f(w + Le_i)\right| \\
&= I_i(f),
\end{align*}
as required. This completes the proof of Conjecture~1.

\end{document}